\documentclass[11pt,reqno]{amsart}
\usepackage{graphicx}
\usepackage{bbm}
\usepackage[colorlinks]{hyperref}
\pdfoutput=1

\title{Diamond Aggregation}
\author{Wouter Kager \and Lionel Levine}

\address{Wouter Kager, Department of Mathematics, VU University Amsterdam, De 
Boelelaan 1081a, 1081\,HV Amsterdam, The Netherlands, 
{\normalfont\url{http://www.few.vu.nl/~wkager}}}
\address{Lionel Levine, Department of Mathematics, Massachusetts Institute of 
Technology, Cambridge, MA 02139, 
{\normalfont\url{http://math.mit.edu/~levine}}}
\keywords{asymptotic shape, growth model, internal diffusion-limited 
aggregation, uniform harmonic measure}
\subjclass[2000]{60K35}

\date{November 10, 2009}

\renewcommand{\Pr}{\mathbbm{P}}						% Probability
\newcommand{\Ex}{\mathbbm{E}}						% Expectation
\newcommand{\I}{\mathbbm{1}}						% Indicator
\newcommand{\Var}{\mathop{\mathrm{Var}}}			% Variance

\newcommand{\Z}{\mathbbm{Z}}						% Integers
\newcommand{\R}{\mathbbm{R}}						% Real numbers

\newcommand{\D}{\mathcal{D}}						% Diamond
\renewcommand{\L}{\mathcal{L}}						% Diamond layer

\newcommand{\Qout}{Q_{\text{out}}}					% Outward kernel
\newcommand{\Qin}{Q_{\text{in}}}					% Inward kernel

\newcommand{\union}{\bigcup}						% Union
\newcommand{\inter}{\bigcap}						% Intersection

\newcommand{\norm}[1]{\left\lVert#1\right\rVert}	% L1-norm of #1
\newcommand{\floor}[1]{\left\lfloor#1\right\rfloor}	% Floor of #1
\newcommand{\ceil}[1]{\left\lceil#1\right\rceil}	% Ceiling of #1

\newcommand{\eps}{\varepsilon}						% Epsilon
\newcommand{\eqindist}{\stackrel{d}{=}}				% Equal in distribution

\newtheorem{theorem}{Theorem}
\newtheorem{lemma}[theorem]{Lemma}

\theoremstyle{remark}
\newtheorem*{remark}{Remark}

\numberwithin{equation}{section}

\begin{document}

\begin{abstract}
	Internal diffusion-limited aggregation is a growth model based on random 
	walk in~$\Z^d$. We study how the shape of the aggregate depends on the law 
	of the underlying walk, focusing on a family of walks in $\Z^2$ for which 
	the limiting shape is a diamond. Certain of these walks---those with a 
	directional bias toward the origin---have at most logarithmic fluctuations 
	around the limiting shape. This contrasts with the simple random walk, 
	where the limiting shape is a disk and the best known bound on the 
	fluctuations, due to Lawler, is a power law. Our walks enjoy a uniform 
	layering property which simplifies many of the proofs.
\end{abstract}

\maketitle

\section{Introduction and main results}
\label{sec:introduction}

Internal diffusion-limited aggregation (internal DLA) is a growth model 
proposed by Diaconis and Fulton~\cite{DF91}. In the original model on~$\Z^d$, 
particles are released one by one from the origin~$o$ and perform simple 
symmetric discrete-time random walks. Starting from the set $A(1) = \{o\}$, 
the clusters $A(i+1)$ for $i\geq1$ are defined recursively by letting the 
$i$-th particle walk until it first visits a site not in $A(i)$, then adding
this site to the cluster. Lawler, Bramson and Griffeath~\cite{LBG92} proved 
that in any dimension $d\geq2$, the asymptotic shape of the cluster~$A(i)$ is 
a $d$-dimensional ball. Lawler~\cite{La95} subsequently showed that the 
fluctuations around a ball of radius~$r$ are at most of order $r^{1/3}$ up to 
logarithmic corrections. Moore and Machta~\cite{MM00} found experimentally 
that the fluctuations appear to be at most logarithmic in~$r$, but there is 
still no rigorous bound to match their simulations. Other studies of internal 
DLA include~\cite{GQ00,BQR03,BB07,LP09b}.

Here we investigate how the shape of an internal DLA cluster depends on the 
law of the underlying random walk. Perhaps surprisingly, small changes in the 
law can dramatically affect the limiting shape. Consider the walk in~$\Z^2$ 
with the same law as simple random walk except on the $x$ and $y$-axes, where 
steps toward the origin are reflected. For example, from a site $(x,0)$ on the 
positive $x$-axis, the walk steps to $(x+1,0)$ with probability~$1/2$ and to 
each of $(x,\pm 1)$ with probability~$1/4$; see Figure~\ref{fig:SimpleKernel}.  
It follows from Theorem~\ref{thm:diamondshape}, below, that when we rescale 
the resulting internal DLA cluster~$A(i)$ to have area~$2$, its asymptotic 
shape as $i\to \infty$ is the diamond
\[
	\D = \{(x,y)\in \R^2 : |x|+|y| \leq 1 \}.
\]

\begin{figure}
	\begin{center}
		\includegraphics{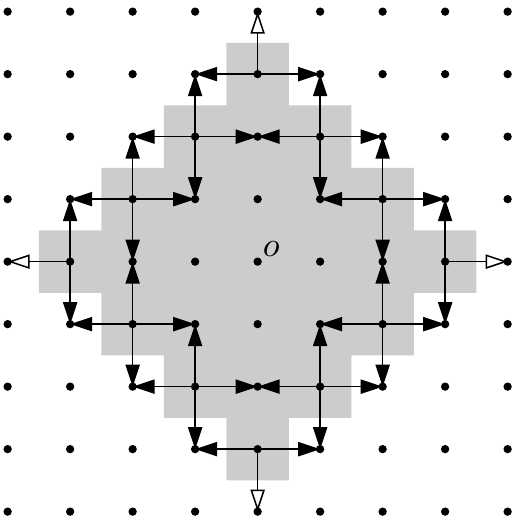}
	\end{center}
	\caption{Example of a uniformly layered walk. The sites enclosed by the 
	shaded area form the diamond~$\D_3$. Only the transition probabilities 
	from layer~$\L_3$ are shown. Open-headed arrows indicate transitions that 
	take place with probability~$1/2$; all the other transitions have 
	probability~$1/4$.}
	\label{fig:SimpleKernel}
\end{figure}
	
In fact, a rather large family of walks produce this diamond as their limiting 
shape. The key property shared by the walks we will consider is that their 
position at any time~$t$ is distributed as a mixture of uniform distributions 
on diamond layers. To define these walks, for $k\geq0$ let
\[
	\L_k := \{x \in \Z^2: \norm{x}=k\}
\]
where for $x=(x_1,x_2)$ we write $\norm{x} = |x_1|+|x_2|$. A \emph{uniformly 
layered walk} is a discrete-time Markov chain on state space $\Z^2$ whose 
transition probabilities $Q(x,y)$ satisfy

\begin{itemize}
	\item[(U1)] $Q(x,y)=0$ if $\norm{y}>\norm{x}+1$;
	\item[(U2)] For all $k\geq 0$ and all $x\in \L_k$, there exists $y\in 
		\L_{k+1}$ with $Q(x,y)>0$;
%	\item[(U2)] For all $k\geq 0$ there exist $x\in \L_k$ and $y\in \L_{k+1}$ 
%		with $Q(x,y)>0$;
%% Note this version is not enough to ensure $\limsup\norm{X(t)}=\infty$, a.s.
	\item[(U3)] For all $k,\ell \geq 0$ and all $y,z \in \L_\ell$,
		\[ \sum_{x\in \L_k} Q(x,y) = \sum_{x\in \L_k} Q(x,z). \]
\end{itemize}

In order to state our main results, let us now give a more precise description 
of the aggregation rules. Set $A(1)=\{o\}$, and let $Y^i(t)$ ($i=1,2,\dotsc$) 
be independent uniformly layered walks with the same law, started from the 
origin. For $i\geq1$, define the stopping times $\sigma^i$ and the growing 
cluster~$A(i)$ recursively by setting
\[
	\sigma^i = \min\{ t\geq0: Y^i(t)\not\in A(i) \}
\]
and
\[
	A(i+1) = A(i) \cup \{ Y^i(\sigma^i) \}.
\]

Now for any real number $r \geq 0$, let
\[
	\D_r := \left\{x \in \Z^2 : \norm{x}\leq r \right\}.
%	\D_n := \L_0 \cup \L_1 \cup \dotsb \cup \L_n
%	\D_n := \union_{0\leq\ell\leq n} \L_\ell
\]
%and for $r \in \R_+$ write $\D_r = \D_{\floor{r}}$. 
We call~$D_r$ the diamond of radius~$r$ in~$\Z^2$. Note that $\D_r = 
D_{\floor{r}}$.
%It is a discrete approximation to the diamond~$n\D$ in~$\R^2$. 
For integer $n \geq 0$, we have $\D_n = \union_{k=0}^n \L_k$. Since~$\#\L_k = 
4k$ for $k\geq 1$, the volume of~$\D_n$ is
\[
	v_n := \#\D_n = 2n(n+1)+1.
\]
Our first result says that the internal DLA cluster of $v_n$~sites based on 
any uniformly layered walk is close to a diamond of radius~$n$. 

\begin{theorem}
	\label{thm:diamondshape}
	For any uniformly layered walk in $\Z^2$, the internal DLA clusters 
	$A(v_n)$ satisfy
	\[
		\Pr\left( \D_{n-4\sqrt{n\log n}} \subset A(v_n) \subset 
		\D_{n+20\sqrt{n\log n}} \text{ eventually} \right) = 1.
	\]
\end{theorem}

Here and throughout this paper \emph{eventually} means ``for all but finitely 
many~$n$.'' Likewise, we will write \emph{i.o.}\ or \emph{infinitely often} to 
abbreviate ``for infinitely many~$n$.''

Our proof of Theorem~\ref{thm:diamondshape} in Section~\ref{sec:general} 
follows the strategy of Lawler~\cite{La95}. The uniform layering property~(U3) 
takes the place of the Green's function estimates used in that paper, and 
substantially simplifies some of the arguments.

\begin{figure}
	\begin{center}
		\includegraphics[width=.46\textwidth]{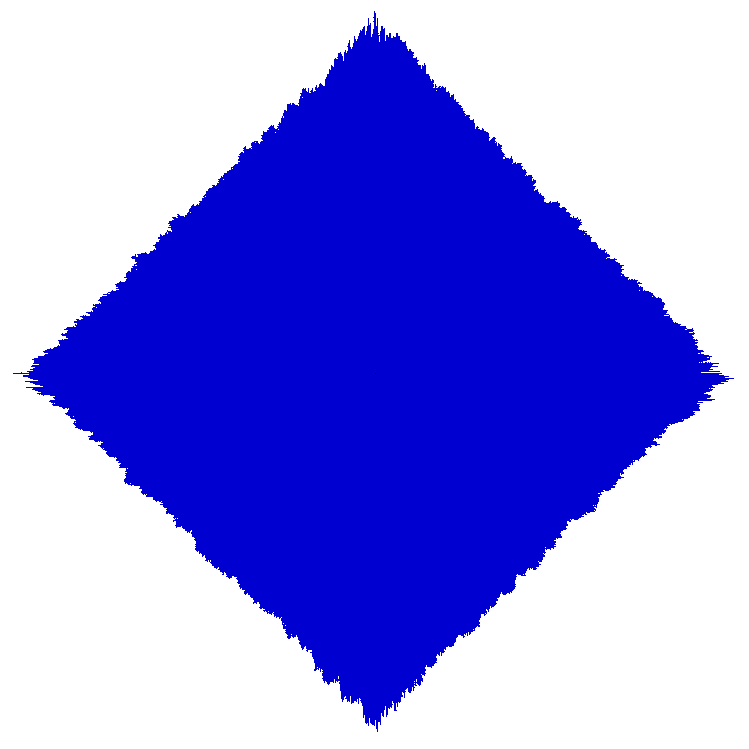}
		\includegraphics[width=.46\textwidth]{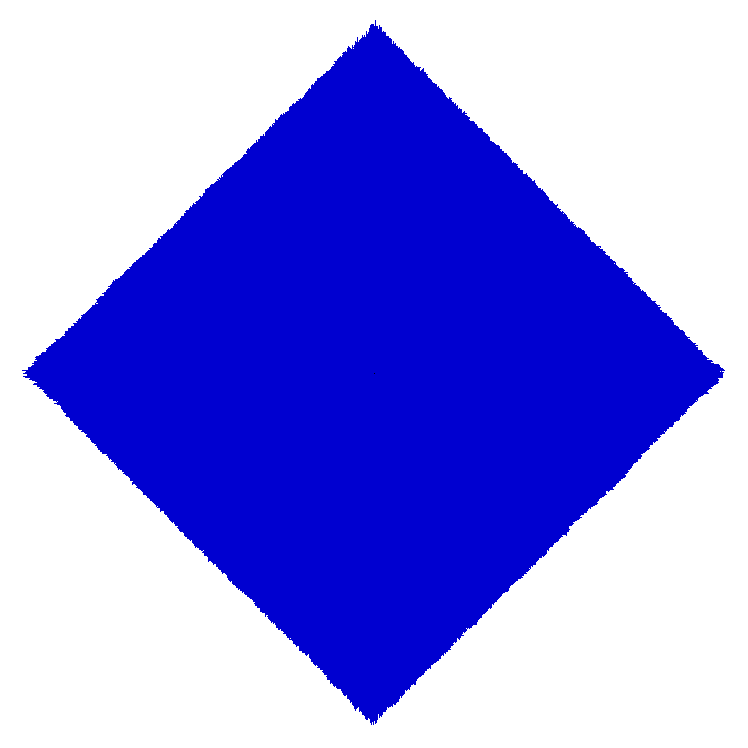}
	\end{center}
	\caption{Internal DLA clusters in $\Z^2$ based on the uniformly layered 
	walk with transition kernel $p\,\Qin + q\,\Qout$. Left: $p=0$, walks are 
	directed outward. Right: $p=1/2$, walks have no directional bias. Each 
	cluster is composed of $v_{350} = 245\,701$ particles.}
	\label{fig:DDiamond}
\end{figure}

Within the family of uniformly layered walks, we study how the law of the walk 
affects the fluctuations of the internal DLA cluster around the limiting 
diamond shape. A natural walk to start with is the outward-directed layered 
walk $X(t)$ satisfying
\[
	\norm{X(t+1)} = \norm{X(t)} + 1
\]
for all $t$. There is a unique such walk satisfying condition (U3) whose 
transition probabilities are symmetric with respect to reflection about the 
axes. It is defined in the first quadrant by
\begin{align}
	\label{eq:Qoutbegin}
	\Qout\bigl( (x,y),(x,y+1) \bigr) &= \frac{y+1/2}{x+y+1}
	&& \text{for $x,y = 1,2,\dotsc$,} \\
	\Qout\bigl( (x,y),(x+1,y) \bigr) &= \frac{x+1/2}{x+y+1}
	&& \text{for $x,y = 1,2,\dotsc$,}
\end{align}
and on the positive horizontal axis by
\begin{align}
	\Qout\bigl( (x,0),(x,\pm1) \bigr) &= \frac{1/2}{x+1}
	&& \text{for $x = 1,2,\dotsc$,} \\
	\label{eq:Qoutend}
	\Qout\bigl( (x,0),(x+1,0) \bigr) &= \frac{x}{x+1}
	&& \text{for $x = 1,2,\dotsc$.}
\end{align}
In the other quadrants~$\Qout$ is defined by reflection symmetry, and at the 
origin we set $\Qout(o,z) = 1/4$ for all $z\in\Z^2$ with $\norm{z}=1$. See 
Figure~\ref{fig:DirKernels}.

\begin{figure}
	\begin{center}
		\includegraphics[width=\textwidth]{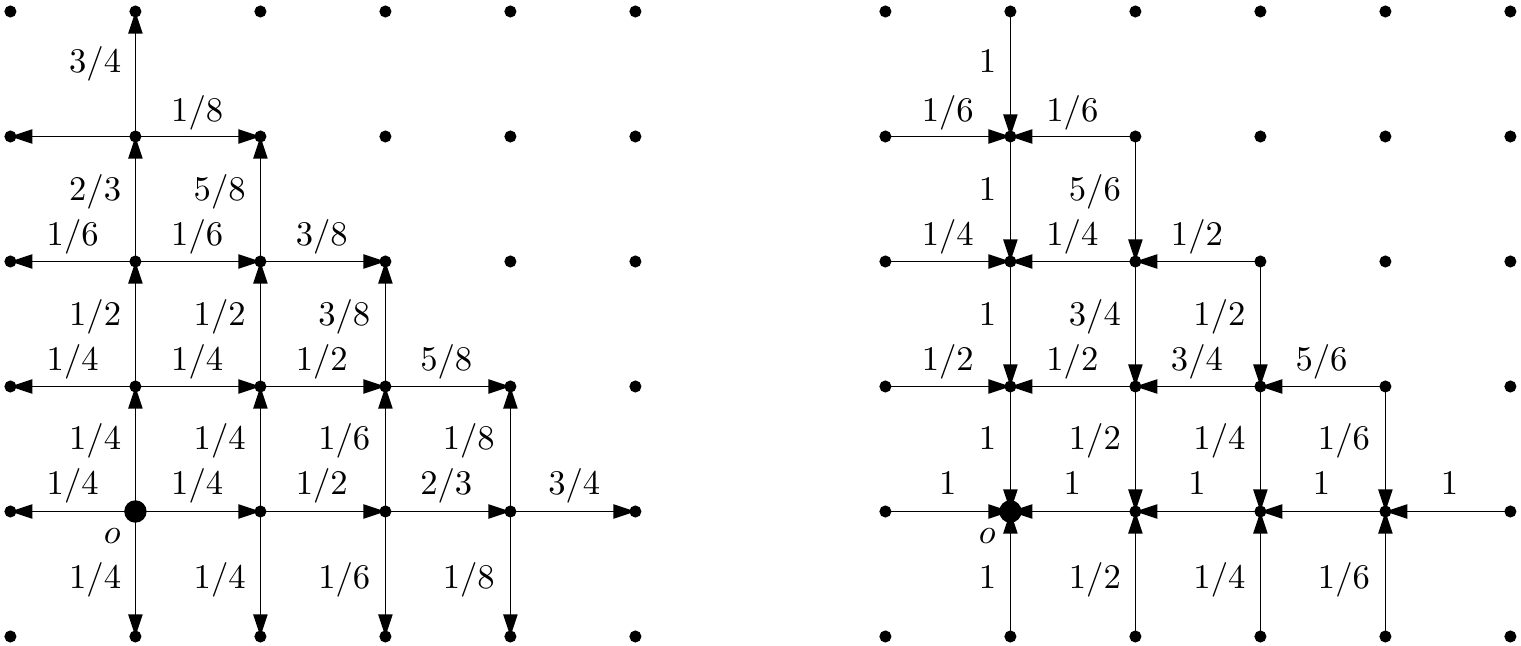}
	\end{center}
	\caption{Left: transition probabilities of the outward directed 
	kernel~$\Qout$. Right: transition probabilities for the inward directed 
	kernel~$\Qin$. The origin $o$ is near the lower-left corner.}
	\label{fig:DirKernels}
\end{figure}

Likewise one can construct a symmetric Markov kernel defining an inward 
directed random walk which remains uniformly distributed on diamond layers. 
This kernel is defined in the first quadrant by
\begin{align}
	\label{eq:Qinbegin}
	\Qin\bigl( (x,y),(x,y-1) \bigr) &= \frac{y-1/2}{x+y-1}
	&& \text{for $x,y = 1,2,\dotsc$,} \\
	\Qin\bigl( (x,y),(x-1,y) \bigr) &= \frac{x-1/2}{x+y-1}
	&& \text{for $x,y = 1,2,\dotsc$,}
\end{align}
and on the positive horizontal axis by
\begin{align}
	\label{eq:Qinend}
	\Qin\bigl( (x,0),(x-1,0) \bigr) &= 1 && \text{for $x = 1,2,\dotsc$.}
\end{align}
Again, the definition extends to the other quadrants by reflection symmetry, 
and is completed by making the origin an absorbing state: $\Qin(o,o) = 1$. See 
Figure~\ref{fig:DirKernels}.

We now choose a parameter $p\in[0,1)$, let $q=1-p$ and define the kernel $Q_p 
:= p\,\Qin + q\,\Qout$. The parameter~$p$ allows us to interpolate between a 
fully outward directed walk at $p=0$ and a fully inward directed walk at 
$p=1$. 

Theorem~\ref{thm:diamondshape} shows that the fluctuations around the limit 
shape are at most of order $\sqrt{n\log n}$ for the entire family of 
walks~$Q_p$. However, one may expect that the true size of the fluctuations 
depends on~$p$. When $p$ is large, particles tend to take a longer time to 
leave a diamond of given radius, affording them more opportunity to fill in 
unoccupied sites near the boundary of the cluster. Indeed, in simulations we 
find that the boundary becomes less ragged as $p$ increases 
(Figure~\ref{fig:DCloseup}). Our next result shows that when $p>1/2$, the 
boundary fluctuations are at most logarithmic in~$n$.

\begin{figure}
	\begin{center}
		\begin{tabular}{cc}
		\includegraphics[width=.445\textwidth]{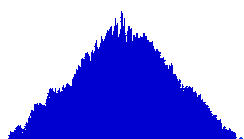} &
		\includegraphics[width=.445\textwidth]{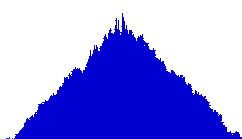} \\
		$p=0$ & $p=1/4$ \\
		\includegraphics[width=.445\textwidth]{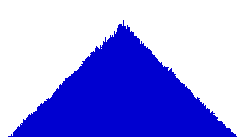} &
		\includegraphics[width=.445\textwidth]{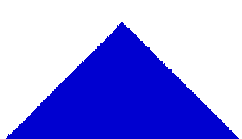} \\
		$p=1/2$ & $p=3/4$
		\end{tabular}
	\end{center}
	\caption{Closeups of the boundary of the diamond. Fluctuations decrease as 
	the directional bias of the walk tends from outward ($p=0$) to inward 
	($p=1$).}
	\label{fig:DCloseup}
\end{figure}

\begin{theorem}
	\label{thm:inward}
	For all $p\in(1/2,1)$, we have
	\[
		\Pr\left( \D_{n-6\log_r n} \subset A(v_n) \subset \D_{n+6\log_r n} 
		\text{ eventually} \right) = 1
	\]
	where the base of the logarithm is $r=p/q$.
\end{theorem}

We believe that for all $p\in[0,1/2)$ the boundary fluctuations are of 
order~$\sqrt{n}$ up to logarithmic corrections, and that therefore an abrupt 
change in the order of the fluctuations takes place at $p=1/2$. At present, 
however, we are able to prove a lower bound on the order of fluctuations only 
in the case $p=0$:

\begin{theorem}
	\label{thm:outward}
	For $p=0$ we have
	\[
		\Pr\left( \D_{n - (1-\eps) \sqrt{2(n\log\log n)/3 }} \not\subset 
		A(v_n) \text{ i.o.} \right) = 1 \qquad \forall\eps>0\phantom{.}
	\]
	and
	\[
		\Pr\left( A(v_n) \not\subset \D_{n + (1-\eps) \sqrt{2(n\log\log n)/3}} 
		\text{ i.o.} \right) = 1 \qquad \forall\eps>0.
	\]
\end{theorem}

Uniformly layered walks are closely related to the walks studied 
in~\cite{Du04, Ka07}. Indeed, the diamond shape of the layers does not play an 
important role in our arguments. A result similar to 
Theorem~\ref{thm:diamondshape} will hold for walks satisfying (U1)--(U3) for 
other types of layers $\L_k$, provided the cardinality $\#\L_k$ grows at most 
polynomially in~$k$. Figure~\ref{fig:Hexagon} shows an example of a walk on 
the triangular lattice satisfying (U1)--(U3) for hexagonal layers. The 
resulting internal DLA clusters have the regular hexagon as their asymptotic 
shape. Blach\`{e}re and Brofferio \cite{BB07} study internal DLA based on 
uniformly layered walks for which $\#\L_k$ grows exponentially, such as simple 
random walk on a regular tree.

\begin{figure}
	\begin{center}
		\includegraphics{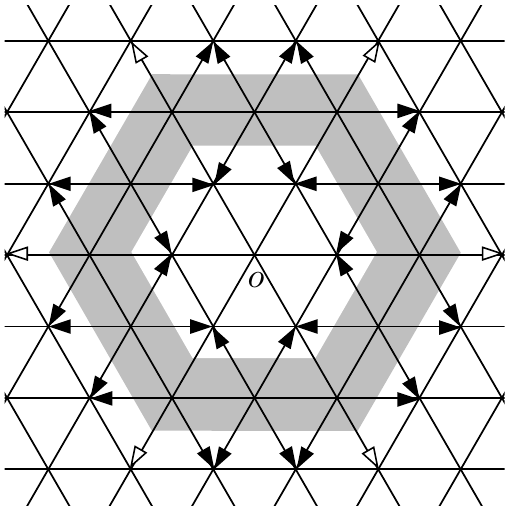} \quad
		\includegraphics[height=2in]{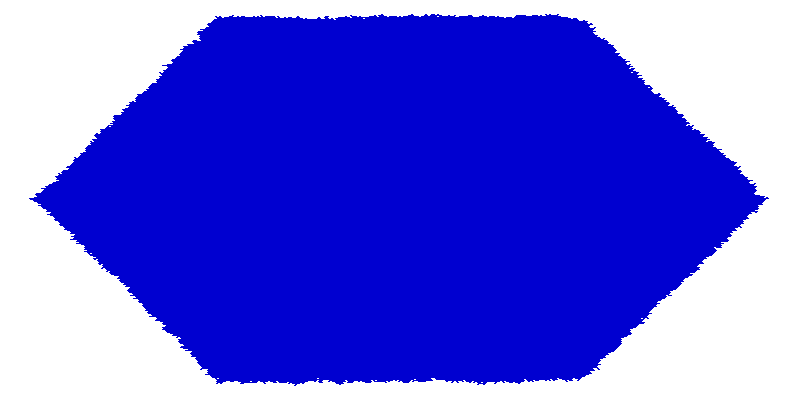}
	\end{center}
	\caption{Left: Example of a uniformly layered walk on the triangular 
	lattice with hexagonal layers. Only transitions from a single (shaded) 
	layer are shown. Open-headed arrows indicate transitions that take place 
	with probability~$1/2$; all the other transitions have probability~$1/4$. 
	Right: An internal DLA cluster of $100\, 000$ particles based on this 
	uniformly layered walk.}
	\label{fig:Hexagon}
\end{figure}

Given how sensitive the shape of an internal DLA cluster is to the law of the 
underlying walk, it is surprising how robust the shape is to other types of 
changes in the model. For example, the particles may perform deterministic 
rotor-router walks instead of simple random walks. These walks depend on an 
initial choice of rotors at each site in $\Z^d$, but for any such choice, the 
limiting shape is a ball. Another variant is the divisible sandpile model, 
which replaces the discrete particles by a continuous amount of mass at each 
lattice site. Its limiting shape is also a ball. These models are discussed 
in~\cite{LP09a}.

The remainder of the paper is organized as follows. 
Section~\ref{sec:preliminaries} explores the properties of uniformly layered 
walks,
section~\ref{sec:abelian} discusses an ``abelian property'' of internal DLA 
which is essential for the proof of Theorem~\ref{thm:diamondshape},
and section~\ref{sec:largedeviations} collects the limit theorems we will use. 
Sections \ref{sec:general}, \ref{sec:inward} and~\ref{sec:outward} are devoted 
to the proofs of Theorems \ref{thm:diamondshape}, \ref{thm:inward} 
and~\ref{thm:outward}, respectively.

\section{Uniformly layered walks}
\label{sec:preliminaries}

Let $\{X(t)\}_{t\geq 0}$ be a uniformly layered walk, that is, a walk on 
$\Z^2$ satisfying properties (U1)--(U3) of the introduction. Write $\nu_k$ for 
the uniform measure on the sites of layer~$\L_k$, and let $\Pr_k$ denote the 
law of the walk started from $X(0) \sim \nu_k$. Likewise, let $\Pr_x$ denote 
the law of the walk started from $X(0)=x$. Consider the stopping times
\begin{align*}
	\tau_z &:= \min\{ t\geq0: X(t) = z \} &&\text{for }z\in\Z^2;\\
	\tau_k &:= \min\{ t\geq0: X(t) \in \L_k \} &&\text{for }k\geq0.
\end{align*}
The key to the diamond shape, as we shall see, is the fact that our random 
walks have the uniform distribution on diamond layers at all fixed times, and 
at the particular stopping times~$\tau_k$. The next lemma shows that 
under~$\Pr_k$, conditionally on $\norm{X(s)}$ for $s \leq t$, the distribution 
of~$X(t)$ is uniform on~$\L_{\norm{X(t)}}$. We remark that the fact that this 
conditional distribution depends only on~$\norm{X(t)}$, and not on 
$\norm{X(s)}$ for $s<t$, implies that~$\norm{X(t)}$ is a Markov chain 
under~$\Pr_k$; see~\cite{RP81}.

\begin{lemma}
	\label{lem:uniformity}
	Fix~$k\geq0$. For all $t\geq 0$ and all sequences of nonnegative integers 
	$k=\ell(0),\dotsc,\ell(t)$ satisfying $\ell(s+1) \leq \ell(s) + 1$ for 
	$s=0,\ldots,t-1$, we have for all $z \in \L_{\ell(t)}$
	\[
		\begin{split}
		\Pr_k \bigl( X(t) = z \bigm| \norm{X(s)} = \ell(s),\ 0\leq s\leq t 
		\bigr)
		&= \frac{1}{\#\L_{\ell(t)}} \\
		&= \Pr_k \bigl( X(t) = z \bigm| \norm{X(t)} = \ell(t) \bigr).
		\end{split}
	\]
\end{lemma}

\begin{proof}
	We prove the first equality by induction on~$t$. The base case $t=0$ is 
	immediate. Write
	\[
		\mathcal{E}_t = \bigl\{ \norm{X(s)} = \ell(s), \, 0\leq s\leq t 
		\bigr\}.
	\]
	By the Markov property and the inductive hypothesis, we have for $t\geq 1$ 
	and any $y \in \L_{\ell(t)}$
	\[
		\begin{split}
		\Pr_k( X(t) = y, \mathcal{E}_t )
		&= \sum_{x\in\L_{\ell(t-1)}} \Pr_k( X(t) = y, X(t-1)=x, 
		\mathcal{E}_{t-1} ) \\
		&= \sum_{x\in\L_{\ell(t-1)}} Q(x,y) \cdot \Pr_k( X(t-1)=x, 
		\mathcal{E}_{t-1} ). \\
		&= \sum_{x\in\L_{\ell(t-1)}} Q(x,y) \cdot \frac{1}{\#\L_{\ell(t-1)}} 
		\cdot \Pr_k( \mathcal{E}_{t-1} ).
		\end{split}
	\]
	By property (U3), the right side does not depend on the choice of $y \in 
	\L_{\ell(t)}$. It follows that
	\[
		\Pr_k( X(t) = z \mid \mathcal{E}_t )
		 =\frac{\Pr_k(X(t)=z, \mathcal{E}_t)}{\sum_{y\in\L_{\ell(t)}} 
		 \Pr_k(X(t)=y,\mathcal{E}_t)}
		 = \frac{1}{\#\L_{\ell(t)}}.
	\]
	By induction this holds for all $t\geq0$ and all sequences $\ell(0), 
	\dotsc, \ell(t)$. Therefore, for fixed $\ell(t)$ and $z\in\L_{\ell(t)}$
	\[
		\begin{split}
		\Pr_k\bigl( X(t) = z \bigr)
		&= \sum_{\ell(0),\dotsc,\ell(t-1)} \Pr_k\bigl( X(t) = z, \norm{X(s)} = 
		\ell(s)\ \forall s\leq t \bigr) \\
		&= \frac{1}{\#\L_{\ell(t)}} \, \sum_{\ell(0),\dotsc,\ell(t-1)} 
		\Pr_k\bigl( \norm{X(s)} = \ell(s)\ \forall s\leq t \bigr) \\
		&= \frac{1}{\#\L_{\ell(t)}} \, \Pr_k\bigl( \norm{X(t)} = \ell(t) 
		\bigr)
		\end{split}
	\]
	which implies
	\[
		\Pr_k\big( X(t) = z \,\big| \norm{X(t)} = \ell(t) \big) = 
		\frac{1}{\#\L_{\ell(t)}}. \qedhere
	\]
\end{proof}

As a consequence of Lemma~\ref{lem:uniformity}, our random walks have the 
uniform distribution on layer~$\ell$ at the stopping time~$\tau_\ell$.

\begin{lemma}
	\label{lem:uniformhitting}
	Fix integers $0 \leq k<\ell$. Then
	\[
		\Pr_k( X(\tau_\ell) = z ) = \frac{1}{\#\L_\ell} = \frac{1}{4\ell} 
		\qquad \text{for every $z\in\L_\ell$}.
	\]
\end{lemma}

\begin{proof}
	Note that property (U2) and Lemma~\ref{lem:uniformity} imply $\tau_\ell < 
	\infty$ almost surely. For $t \geq 0$ we have
	\[
		\{\tau_\ell = t\} = \union_{\ell_0,\dotsc,\ell_t}
		\bigl\{ \norm{X(s)} = \ell_s, 0\leq s\leq t \bigr\},
	\]
	where the union is over all sequences of nonnegative integers $\ell_0, 
	\ell_1, \dotsc, \ell_t$ with $\ell_0=k$ and $\ell_t=\ell$, such that 
	$\ell_{s+1} \leq \ell_s + 1$ and $\ell_s\neq\ell$ for all 
	$s=0,1,\ldots,t-1$. Writing $\mathcal{E}_{\ell_0, \dotsc, \ell_t}$ for the 
	disjoint events in this union, it follows that
	\[
		\begin{split}
		\Pr_k( X(\tau_\ell) = z )
		&= \sum_{t\geq0} \Pr_k( X(t) = z, \,\tau_\ell = t ) \\
		&= \sum_{t\geq0} \sum_{\ell_0,\dotsc,\ell_t}
			\Pr_k( X(t) = z, \,\mathcal{E}_{\ell_0,\dotsc,\ell_t} ) \\
		&= \sum_{t\geq0} \sum_{\ell_0,\dotsc,\ell_t}
			\Pr_k( X(t) = z \mid \mathcal{E}_{\ell_0,\dotsc,\ell_t} )
			\Pr_k( \mathcal{E}_{\ell_0,\dotsc,\ell_t} ).
		\end{split}
	\]
	Since $\sum_{t\geq0} \Pr_k(\tau_\ell=t)=1$, the result follows from 
	Lemma~\ref{lem:uniformity}.
\end{proof}

The previous lemmas show that one can view our random walks as walks that move 
from layer to layer on the lattice, while remaining uniformly distributed on 
these layers. This idea can be formalized in terms of an intertwining relation 
between our two-dimensional walks and a one-dimensional walk that describes 
the transitions between layers, an idea explored in~\cite{Du04, Ka07} for 
closely related random walks in wedges. This approach is particularly useful 
for computing properties of the Green's function.

Next we calculate some hitting probabilities for the walk with transition 
kernel $Q_p = p\,Q_{in} + q\,Q_{out}$ defined in the introduction; we will use 
these in the proof of Theorem~\ref{thm:inward}. We start with the probability 
of visiting the origin before leaving the diamond of radius~$n$. By the 
definition of~$Q_p$, this probability depends only on the layer on which the 
walk is started, not on the particular starting point on that layer. That is, 
if $0<\ell<n$, then $\Pr_x( \tau_o <\tau_n ) = \Pr_\ell( \tau_o <\tau_n )$ for 
all $x\in\L_\ell$, since at every site except the origin, the probability to 
move inward is~$p$ and the probability to move outward is~$q$. This leads to 
the following well-known gambler's ruin calculation (see, e.g.,~\cite[\S 7]{Bi95}).

\begin{lemma}
	\label{lem:hitorigin}
	Let $0<\ell<n$ and $x\in\L_\ell$. If $p\neq q$, then
	\[
		\Pr_x(\tau_o<\tau_n) = \Pr_\ell(\tau_o<\tau_n)
		= \frac{r^n-r^\ell}{r^n-1}
	\]
	where $r= p/q$. If $p=q=1/2$, then
	\[
		\Pr_x(\tau_o<\tau_n) = \Pr_\ell(\tau_o<\tau_n) = \frac{n-\ell}{n}.
	\]
\end{lemma}

%\begin{proof}
%	Set $T=\tau_o \wedge \tau_n$. Since the probability to move inward is~$p$ 
%	and the probability to move outward is~$q$ at each step, for $p\neq q$ the 
%	process $r^{\norm{X(t\wedge T)}}$ is a martingale with bounded increments.  
%	By optional stopping,
%	\[
%		r^\ell = \Ex_\ell\bigl( r^{\norm{X(T)}} \bigr)
%		= \Pr_\ell(\tau_o<\tau_n) + \bigl[ 1-\Pr_\ell(\tau_o<\tau_n) \bigr]\, 
%		r^n,
%	\]
%	from which we obtain $\Pr_\ell(\tau_o<\tau_n)$. For $p=q=1/2$, the process 
%	$\norm{X(t\wedge T)}$ is itself a martingale, hence
%	\[
%		\ell = \Ex_\ell\bigl( \norm{X(T)} \bigr)
%		= \bigl[ 1-\Pr_\ell(\tau_o<\tau_n) \bigr]\, n. \qedhere
%	\]
%\end{proof}

Next we bound the probability that the inward-biased walk ($p>1/2$) exits the 
diamond $\D_{n-1}$ before hitting a given site $z\in \D_{n-1}$.

\begin{lemma}
	\label{lem:youcantavoidz}
	Write $r= p/q$. For $p\in(1/2,1)$, if $z\in\L_k$ for $0<k<n$, then
	\[
		\Pr_o(\tau_z \geq \tau_n) < (4k-1)r^{k-n}.
	\]
\end{lemma}

\begin{proof}
	Let $T_0 = 0$ and for $i\geq 1$ consider the stopping times
	\begin{align*}
		U_i	&= \min\{ t>T_{i-1} : X(t)\in\L_k \}; \\
		T_i &= \min\{ t>U_i : X(t)=o \}.
	\end{align*}
	Let $M = \max\{ i: U_i<\tau_n \}$. For any integer $m\geq 1$ and any 
	$x_1,\dotsc,x_m \in \L_k$, we have by the strong Markov property
	\begin{multline*}
		\Pr_o( M=m,\ X(U_1)=x_1, \dotsc, X(U_m)=x_m ) \\
		= \prod_{i=1}^{m-1} \bigl[ \Pr_o( X(\tau_k)=x_i ) \, \Pr_{x_i}( 
		\tau_o<\tau_n ) \bigr] \cdot \Pr_o( X(\tau_k)=x_m ) \, \Pr_{x_m}( 
		\tau_n<\tau_o ).
	\end{multline*}
	By Lemma~\ref{lem:uniformhitting}, $\Pr_o( X(\tau_k)=x_i ) = 1/4k$ for 
	each $x_i\in\L_k$. Moreover, by Lemma~\ref{lem:hitorigin} we have for any 
	$x\in\L_k$
	\[
		\Pr_x(\tau_n<\tau_0) = \frac{r^k-1}{r^n-1} < r^{k-n},
	\]
	where we have used the fact that $r=p/q>1$. Hence
	\[
		\Pr_o( M=m,\ X(U_i)\neq z\ \forall i\leq m )
		< r^{k-n} \left( 1-\frac{1}{4k} \right)^m.
	\]
	Since the event $\{\tau_z\geq\tau_n\}$ is contained in the event $\{ 
	X(U_i)\neq z\ \forall i\leq M \}$, we conclude that
	\[
		\begin{split}
		\Pr_o(\tau_z\geq\tau_n)
		&= \sum_{m\geq1} \Pr_o(M=m,\ \tau_z\geq \tau_n) \\
		&\leq \sum_{m\geq1} \Pr_o( M=m,\ X(U_i)\neq z\ \forall i\leq m ) \\
		&< \sum_{m\geq1} r^{k-n} \left( 1-\frac{1}{4k} \right)^m \\
		&= (4k-1)r^{k-n}. \qedhere
		\end{split}
	 \]
\end{proof}

\section{Abelian property}
\label{sec:abelian}

In this section we discuss an important property of internal DLA 
discovered by Diaconis and Fulton~\cite[Theorem~4.1]{DF91}, which gives some freedom in how the clusters $A(i)$ are constructed.  We will use this property in
the proof of Theorem~\ref{thm:diamondshape}. It was also used 
in~\cite{La95}.
%To form the cluster~$A(i)$, 
Instead of performing~$i$ random 
walks one at a time in sequence, start with~$i$ particles at the 
origin. At each time step, choose a site occupied by more than one particle, 
and let one particle take a single random walk step from that site. The 
abelian property says that regardless of these choices, the final set of~$i$ 
occupied sites has the same distribution as the cluster~$A(i)$.

This property is not dependent on the law of the random walk, and in fact 
holds deterministically in a certain sense. Suppose that at each site $x \in 
\Z^2$ we place an infinite stack of cards, each labeled by a site in~$\Z^2$. A 
\emph{legal move} consists of choosing a site~$x$ which has at least two 
particles, burning the top card at~$x$, and then moving one particle from~$x$ 
to the site labeled by the card just burned. A finite sequence of legal moves 
is \emph{complete} if it results in a configuration in which each site has at 
most one particle.

\begin{lemma}[Abelian property]
	\label{abelianproperty}
	For any initial configuration of particles on~$\Z^2$,
	if there is a complete finite sequence of legal moves, then any sequence 
	of legal moves is finite, and any complete sequence of legal moves results 
	in the same final configuration.
\end{lemma}

In our setting, the cards in the stack at~$x$ have i.i.d.\ labels with 
distribution $Q(x,\cdot)$. Starting with $i$~particles at the origin, one way 
to construct a complete sequence of legal moves is to let each particle in 
turn perform a random walk until reaching an unoccupied site. The resulting 
set of occupied sites is the internal DLA cluster~$A(i)$.  By the abelian 
property, any other complete sequence of legal moves yields the same 
cluster~$A(i)$.

For the proof of Theorem~\ref{thm:diamondshape}, it will be useful to define 
generalized internal DLA clusters for which not all walks start at the origin.  
Given a (possibly random) sequence $x_1,x_2,\ldots \in \Z^2$, we define the 
clusters $A(x_1,\ldots,x_i)$ recursively by setting $A(x_1)=\{x_1\}$, and
\[
	A(x_1,\ldots,x_{i+1}) = A(x_1,\ldots,x_i) \cup \{Y^i(\sigma^i)\},
	\qquad i\geq 1,
\]
where the $Y^i$ are independent uniformly layered walks started from 
$Y^i(0)=x_{i+1}$, and
\[
	\sigma^i = \min \{t \geq 0 \,:\, Y^i(t) \notin A(x_1,\ldots,x_i) \}.
\]
When $x_1 = \cdots = x_i =o$ we recover the usual cluster~$A(i)$.  

%Given random sets $A,B \subset \Z^2$, we write $A \eqindist B$ to mean that 
%$A$ and $B$ have the same distribution.  By the abelian property, if $x'_1,\ldots,x'_i$ is a (possibly random) permutation of $x_1,\ldots,x_i$, then
%	\[ A(x'_1,\ldots,x'_i) \eqindist A(x_1,\ldots,x_i). \]
%Note that this equality holds only in distribution, due to our convention for ordering the walks.

The next lemma gives conditions under which two such generalized clusters can 
be coupled so that one is contained in the other.  Let $x_1, \ldots, x_r$ and $y_1, \ldots, y_s$ be random points in $\Z^2$.  For $z \in \Z^2$, let 
	\begin{align*}  N_z &= \#\{i \leq r \,:\, x_i=z\} \\
	  		\tilde{N}_z &= \#\{j \leq s \,:\, y_j=z\} \end{align*}
and consider the event
	\[ \mathcal{E} = \bigcap_{z \in \Z^2} \big\{ N_z \leq \tilde{N}_z \big\}. \]

\begin{lemma}[Monotonicity]
	\label{monotonicity}
	There exists a random set~$A'$ with the same distribution as $A(y_1,\ldots,y_s)$, such that $\mathcal{E} \subset \big\{ A(x_1,\ldots,x_r) \subset A' \big\}$.
%	Let $x_1, \ldots, x_r, y_1, \ldots, y_s \in \Z^2$. If for all $z \in \Z^2$ 
%	we have
%	\[ \#\{i \leq r \,:\, x_i=z\} \leq \#\{j \leq s \,:\, y_j=z\} \]
%	then there exists a random set~$A'$ with the same distribution as 
%	$A(y_1,\ldots,y_s)$, such that
%	\[ A(x_1,\ldots,x_r) \subset A'. \]
\end{lemma}

The proof follows directly from the abelian property: since the distribution 
of $A(y_1,\ldots,y_s)$ does not depend on the ordering of the points 
$y_1,\ldots,y_s$, we can take
	\[ A' = \begin{cases} A(y'_1, \ldots, y'_s) & \mbox{ on } \mathcal{E} \\ 
					A(y_1,\ldots,y_s) & \mbox{ on } \mathcal{E}^c \end{cases} \]
where $y'_1,\ldots,y'_s$ is a (random) permutation of $y_1,\ldots,y_s$ such that 
$y'_i=x_i$ for all $i\leq r$. 

%By permuting $y_1,\dots,y_n$ so that $y_i=x_i$ for $i\leq m$, we obtain a 
%cluster with the same distribution as $A(y_1,\ldots,y_n)$ and containing 
%$A(x_1,\ldots,x_m)$ as a subset.

\section{Sums of independent random variables}
\label{sec:largedeviations}

We collect here a few standard results about sums of independent random 
variables. First we consider large deviation bounds for sums of independent 
indicators, which we will use several times in the proofs of Theorems 
\ref{thm:diamondshape} and~\ref{thm:inward}. Let $S$ be a finite sum of 
independent indicator random variables. We start with simple Chernoff-type 
bounds based on the inequality
\[
	\Pr(S\geq b) \leq e^{-tb} \, \Ex\left( e^{tS} \right).
\]
There are various ways to give an upper bound on the right side when the 
summands of~$S$ are i.i.d.\ indicators; see for 
example~\cite[Appendix~A]{AS92}. These bounds extend to the case of 
independent but not necessarily identically distributed indicators by an 
application of Jensen's inequality, leading to the following 
bounds~\cite[Theorems 1 and~2]{Ja02}:

\begin{lemma}[Chernoff bounds]
	\label{lem:Chernoff}
	Let $S$ be a finite sum of independent indicator random variables. For all 
	$b\geq0$,
	\begin{align*}
		\Pr(S\geq\Ex S+b)
		&\leq \exp\left( -\frac{1}{2} \frac{b^2}{\Ex S+b/3} \right), \\
		\Pr(S\leq\Ex S-b)
		&\leq \exp\left( -\frac{1}{2} \frac{b^2}{\Ex S} \right).
	\end{align*}
\end{lemma}

Next we consider limit theorems for sums of independent random variables, 
which we will use in the proof of Theorem~\ref{thm:outward}. For $\{ X_n 
\}_{n\geq 1}$ a sequence of independent random variables satisfying 
$\Ex|X_i|^3 < \infty$, we define
\begin{align}
	B_n &= \sum_{1\leq i\leq n} \Var(X_i), \label{eq:Bn} \\
	L_n &= B_n^{-3/2} \sum_{1\leq i\leq n} \Ex|X_i-\Ex X_i|^3. \label{eq:Ln}
\end{align}
It is well known that the partial sums
\begin{equation}
	\label{eq:Sn}
	S_n = \sum_{1\leq i\leq n} X_i
\end{equation}
satisfy the Central Limit Theorem when $L_n\to0$; this is a special case of 
Lyapunov's condition. We are interested in the rate of convergence. Let
\begin{equation}
	\label{eq:Deltan}
	\Delta_n = \sup_{x\in\R} \left| \Pr\left( S_n-\Ex S_n < x\sqrt{B_n} 
	\right) - \Phi(x) \right|,
\end{equation}
where $\Phi$ is the standard normal distribution function. Esseen's inequality 
(see \cite[Introduction, equation~(6)]{Es45} and \cite[Chapter~I]{PS00}) gives 
a bound on~$\Delta_n$ in terms of~$L_n$. This bound can be used to verify the 
conditions given by Petrov~\cite[Theorem~1]{Pe66} (see also 
\cite[Chapter~I]{PS00}), under which the partial sums~$S_n$ satisfy the Law of 
the Iterated Logarithm.

\begin{lemma}[Esseen's inequality]
	\label{lem:Esseen}
	Let $X_1,\dotsc,X_n$ be independent and such that $\Ex|X_i|^3 < \infty$, 
	and define $B_n$, $L_n$, $S_n$ and~$\Delta_n$ by 
	\eqref{eq:Bn}--\eqref{eq:Deltan}. Then
	\[
		\Delta_n \leq 7.5 \cdot L_n.
	\]
\end{lemma}

\begin{lemma}[Petrov's theorem]
	\label{lem:Petrov}
	Let $\{X_i\}_{i\geq1}$ be a sequence of independent random variables with 
	finite variances, and define $B_n$, $S_n$ and~$\Delta_n$ by \eqref{eq:Bn}, 
	\eqref{eq:Sn} and~\eqref{eq:Deltan}. If, as $n\to\infty$,
	\[
		B_n\to\infty, \quad \frac{B_{n+1}}{B_n}\to1 \quad\text{and}\quad 
		\Delta_n = O\left( \frac{1}{(\log B_n)^{1+\delta}} \right) \text{ for 
		some $\delta>0$},
	\]
	then
	\[
		\Pr\left( \limsup_{n\to\infty}
			\frac{S_n-\Ex S_n}{\sqrt{2B_n \log\log B_n}} = 1
		\right) = 1.
	\]
\end{lemma}
\vspace{0mm}

\section{Proof of Theorem~\ref{thm:diamondshape}}
\label{sec:general}

We control the growth of the cluster $A(i)$ by relating it to two modified 
growth processes, the \emph{stopped process} $S(i)$ and the \emph{extended 
process} $E(i)$. In the stopped process, particles stop walking when they hit 
layer~$\L_n$, even if they have not yet found an unoccupied site. More 
formally, let $S(1) = \{o\}$, and define the stopping times~$\sigma_S^i$ and 
clusters $S(i)$ for $i\geq1$ recursively by
\begin{equation*}
%	\label{eq:stoppedaggtime}
	\sigma_S^i
	= \min\{ t\geq0: Y^i(t) \in \L_n \cup S(i)^c \}
\end{equation*}
and
\begin{equation}
	\label{eq:stoppedcluster}
	S(i+1)= S(i) \cup \{ Y^i(\sigma_S^i) \}.
\end{equation}
Here $Y^i(t)$ for $i=1,2,\ldots$ are independent uniformly layered walks 
started from the origin in $\Z^2$, all having the same law. Note that 
$S(i+1)=S(i)$ on the event that the walk~$Y^i$ hits layer $\L_n$ before 
exiting the cluster $S(i)$. By the abelian property, 
Lemma~\ref{abelianproperty}, we have $S(i) \subset A(i)$. Indeed, $A(i)$ can 
be obtained from $S(i)$ by letting all but one of the particles stopped at 
each site in $\L_n$ continue walking until reaching an unoccupied site.

The extended process $E(i)$ is defined by starting with every site in the 
diamond~$\D_n$ occupied, and letting each of $i$~additional particles in turn 
walk from the origin until reaching an unoccupied site. More formally, let 
$E(0)=\D_n$, and for $i\geq 0$ define
\[
	\sigma_E^i = \min\{ t\geq0: Y^{v_n+i}(t) \notin E(i) \}
\]
and
\[
	E(i+1)= E(i) \cup \{ Y^{v_n+i}(\sigma_E^i) \}.
\]

An outline of the proof of Theorem~\ref{thm:diamondshape} runs as follows. We 
first show in Lemma~\ref{lem:Sbound} that the stopped cluster $S(v_n)$ 
contains a large diamond with high probability. Since the stopped cluster is 
contained in $A(v_n)$, the inner bound of Theorem~\ref{thm:diamondshape} 
follows. The proof of the outer bound proceeds in three steps. 
Lemma~\ref{lem:Nzbound} shows that the particles that stop in layer~$\L_n$ 
during the stopped process cannot be too bunched up at any single site $z \in 
\L_n$. We then use this to argue in Lemma~\ref{lem:AinsideE} that with high 
probability, $A(v_n)$ is contained in a suitable extended cluster $E(m)$. 
Finally, Lemma~\ref{lem:Ebound} shows that this extended cluster is contained 
in a slightly larger diamond.
 
A notable feature of the argument (also present in \cite{La95}) is that the 
proof of the outer bound relies on the inner bound: Lemma~\ref{lem:Sbound} is 
used in the proof of Lemma~\ref{lem:Nzbound}. This dependence is responsible 
for the larger constant in the outer bound of Theorem~\ref{thm:diamondshape}.  
It would be interesting to know whether this asymmetry is merely an artifact 
of the proof, or whether the outer fluctuations are really larger than the 
inner fluctuations.

We introduce an auxiliary collection of walks that will appear in the proofs. 
Let $\{Y^x(t): x\in \Z^2\}$ be independent walks with the same transition 
probabilities as~$Y^1$, which are independent of the~$Y^i$, and which start 
from~$Y^x(0)=x$. Now for $i=1,\ldots,v_n-1$ define
\[
	X^i(t) = \begin{cases}
		Y^i(t) & \text{for $0\leq t\leq\sigma_S^i$}, \\
		Y^{Y^i(\sigma_S^i)}(t-\sigma_S^i) & \text{for $t>\sigma_S^i$}.
	\end{cases}
\]
Note that replacing the walks $Y^i$ with~$X^i$ in~\eqref{eq:stoppedcluster} 
has no effect on the clusters~$S(i)$. Finally, for $i\geq v_n$ we set 
$X^i(t)=Y^i(t)$ for all $t\geq 0$.

We associate the following stopping times with the auxiliary walks~$Y^x(t)$:
\begin{align*}
	\tau^x_z &:= \min\{ t\geq0: Y^x(t) = z \} &&\text{for $z\in\Z^2$}; \\
	\tau^x_k &:= \min\{ t\geq0: Y^x(t) \in \L_k \} &&\text{for $k\geq 0$}.
\end{align*}
Likewise, let
\begin{align*}
	\tau_z^i &:= \min\{ t\geq0: X^i(t) = z \} &&\text{for $z\in\Z^2$};\\
	\tau_k^i &:= \min\{ t\geq0: X^i(t) \in \L_k \} &&\text{for $k\geq0$}.
\end{align*}

\begin{lemma}
	\label{lem:Sbound}
	There exists $n_0$ such that for all uniformly layered walks and all 
	$n\geq n_0$
	\begin{equation} \label{eq:innerdiamond}
		\Pr\left( \D_{n-4\sqrt{n\log n}} \not\subset S(v_n) \right)
		< 6n^{-2}.
	\end{equation}
\end{lemma}

\begin{remark}
	To avoid referring to too many unimportant constants, for the rest of this 
	section we will take the phrase ``for sufficiently large $n$,'' and its 
	variants, to mean that a single bound on~$n$ applies to all uniformly 
	layered walks.  
	%Thus the next lemma should be read as: There exists $n_0$ such that 
	%(\ref{eq:innerdiamond}) holds for all uniformly layered walks and all 
	%$n\geq n_0$.
\end{remark}

\begin{proof}
	For $z\in \D_{n-1}$, write
	\[
		\mathcal{E}_z(v_n) = \bigcap_{i=1}^{v_n-1} \left\{ \sigma_S^i < 
		\tau^i_z \right\}
	\]
	for the event that the site~$z$ does not belong to the stopped cluster 
	$S(v_n)$. We want to show that $\Pr\bigl( \mathcal{E}_z(v_n) \bigr)$ is 
	very small when $z$ is taken too deep inside~$\D_n$. To this end, let 
	$\ell = \norm{z}$, and consider the random variables
	\begin{align*}
		N_z &= \sum\nolimits_{0<i<v_n} \I\{ \tau^i_z \leq \sigma_S^i \}, \\
		M_z &= \sum\nolimits_{0<i<v_n} \I\{ \tau^i_z = \tau^i_{\ell} \}, \\
		L_z &= \sum\nolimits_{0<i<v_n} \I\{\sigma_S^i<\tau^i_z=\tau^i_\ell\}.
	\end{align*}
	Then $\mathcal{E}_z(v_n) = \{N_z=0\}$. Since $N_z\geq M_z-L_z$, we have 
	for any real number~$a$
	\begin{equation}
	\label{eq:geninfundbound}
	\begin{split}
		\Pr\bigl( \mathcal{E}_z(v_n) \bigr)
		= \Pr(N_z=0) &\leq \Pr(M_z\leq a \text{ or } L_z\geq a) \\
		&\leq \Pr(M_z\leq a) + \Pr(L_z\geq a).
	\end{split}
	\end{equation}
	Our choice of~$a$ will be made below. Note that $M_z$ is a sum of i.i.d.\ 
	indicator random variables, and by Lemma~\ref{lem:uniformhitting},
	\begin{equation}
	\label{eq:EM_z}
		\Ex M_z
		= 2n(n+1) \, \Pr_o( X(\tau_{\ell}) = z )
		= \frac{1}{2} \, \frac{n(n+1)}{\ell}.
	\end{equation}
	The summands of~$L_z$ are not independent. Following~\cite{LBG92}, 	
	however, we can dominate~$L_z$ by a sum of independent indicators as 
	follows. By property~(U1), a uniformly layered walk cannot exit the 
	diamond $\D_{\ell-1}$ without passing through layer~$\L_\ell$, so the 
	event $\{ \sigma_S^i < \tau^i_z = \tau^i_\ell \}$ is contained in the 
	event $\{ X^i(\sigma_S^i)\in \D_{\ell-1} \}$. Hence
	\[
		\begin{split}
		L_z
		&= \sum_{0<i<v_n} \I\left\{ X^i(\sigma_S^i)\in \D_{\ell-1},\; 
		\tau^{X^i(\sigma_S^i)}_z = \tau^{X^i(\sigma_S^i)}_\ell \right\} \\
		&\leq \sum_{x\in \D_{\ell-1}-\{o\}} \I\{ \tau_z^x = \tau_\ell^x \}
		=: \tilde{L}_z
		\end{split}
	\]
	where we have used the fact that the locations $X^i(\sigma_S^i)$ inside 
	$D_{\ell-1}$ where particles attach to the cluster are distinct. Note that 
	$\tilde{L}_z$ is a sum of independent indicator random variables. To 
	compute its expectation, note that for every $0<k<\ell$, by 
	Lemma~\ref{lem:uniformhitting}
	\[
		\sum_{x\in \L_k} \Pr_x( X(\tau_{\ell}) = z )
		= 4k \, \Pr_k( X(\tau_{\ell}) = z ) = \frac{k}{\ell},
	\]
	hence
	\begin{equation}
		\label{eq:EL_z}
		\Ex\tilde{L}_z = \sum_{k=1}^{\ell-1} \frac{k}{\ell} = 
		\frac{\ell-1}{2}.
	\end{equation}

	Now set $a = \tfrac{1}{2} (\Ex M_z+\Ex \tilde{L}_z)$, and let
	\[
		b=\frac{\Ex M_z-\Ex \tilde{L}_z}{2} > \frac{n^2 - \ell^2}{4\ell}
	\]
	where the inequality follows from \eqref{eq:EM_z} and~\eqref{eq:EL_z}.  
	Since $a = \Ex M_z-b = \Ex \tilde{L}_z+b$, we have by 	
	Lemma~\ref{lem:Chernoff}
	\[
		\Pr(\tilde{L}_z\geq a)
		\leq \exp\left( -\frac{1}{2} \, \frac{b^2}{\Ex\tilde{L}_z+b/3} \right) 
		\leq \exp\left( -\frac{1}{2} \, \frac{b^2}{\Ex M_z} \right)
	\]
	and
	\[
		\begin{split}
		\Pr(M_z\leq a)
		&\leq \exp\left(-\frac{1}{2} \, \frac{b^2}{\Ex M_z} \right) \\
		&< \exp \left( -\frac12 \frac{(n^2 - \ell^2)^2}{16 \ell^2} 
		\frac{2\ell}{n(n+1)} \right) \\
		&\leq \exp \left( -\frac{1}{16} \frac{(n^2-\ell^2)^2}{n^3} \right)
		\end{split}
	\]
	where in the last line we have used $\ell \leq n-1$.  Since $L_z \leq 
	\tilde{L}_z$, we obtain from~\eqref{eq:geninfundbound}
	\[
		\begin{split}
		\Pr\bigl( \mathcal{E}_z(v_n) \bigr)
		&\leq \Pr(M_z\leq a) + \Pr(\tilde{L}_z\geq a) \\
		&< 2\exp\left( -\frac{1}{16} \, \frac{(n^2-\ell^2)^2}{n^3} \right).
		\end{split}
	\]
	Writing $\ell = n-\rho$, with $\rho\geq \ceil{4\sqrt{n\log n}}$, we obtain 
	for sufficiently large $n$
	\[
		\begin{split}
		\Pr\bigl( \mathcal{E}_z(v_n) \bigr)
		&< 2\exp\left( -\frac{1}{16} \frac{\rho^2(2n-\rho)^2}{n^3} \right) \\
		&\leq 2\exp\left( -\frac{\rho^2}{4n} + \frac{\rho^3}{4n^2} \right) \\
		&\leq 3n^{-4}.
		\end{split}
	\]
	We conclude that for $n$ sufficiently large
	\[
		\Pr\left( \D_{n-4\sqrt{n\log n}} \not\subset S(v_n) \right)
		\leq \sum_{z\in\D_{n-4\sqrt{n\log n}}}
			 \Pr\bigl( \mathcal{E}_z(v_n) \bigr)
		< 6n^{-2}. \qedhere
	\]
\end{proof}

Turning to the outer bound of Theorem~\ref{thm:diamondshape}, the first step 
is to bound the number
\begin{equation}
	\label{eq:Nz}
	N_z := \sum\nolimits_{0<i<v_n} \I\{ \sigma_S^i = \tau_z^i \}
\end{equation}
of particles stopping at each site $z \in \L_n$ in the course of the stopped 
process. To get a rough idea of the order of~$N_z$, note that according to 
Lemma~\ref{lem:Sbound}, with high probability, at least $v_{n-4\sqrt{n\log 
n}}$ of the $v_n$~particles find an occupied site before hitting layer~$\L_n$.  
The number of particles remaining is of order $n^{3/2} \sqrt{\log n}$. If 
these remaining particles were spread evenly over~$\L_n$, then there would be 
order $\sqrt{n \log n}$ particles at each site $z \in \L_n$. The following 
lemma shows that with high probability, all of the $N_z$ are at most of this 
order.

\begin{lemma}
	\label{lem:Nzbound}
	If $n$ is sufficiently large, then
	\[
		\Pr\biggl( \union_{z\in\L_n} \left\{ N_z > 7\sqrt{n\log n} \right\} 
		\biggr) < 13n^{-5/4}.
	\]
\end{lemma}

\begin{proof}
	For $z\in\L_n$, define
	\begin{align*}
		M_z &= \sum\nolimits_{0<i<v_n} \I\{\tau^i_z = \tau^i_n \}, \\
		L_z &= \sum\nolimits_{0<i<v_n} \I\{ \sigma_S^i<\tau^i_z=\tau^i_n \},
	\end{align*}
	so that $N_z = M_z-L_z$. Write $\eta = \sqrt{n\log n}$ and $\rho = 
	\ceil{4\eta}$, and let
	\[
		\tilde{L}_z = \sum_{y\in\D_{n-\rho}-\{o\}} \I\{ \tau_z^y=\tau_n^y \}.
	\]
	Note that $\tilde{L}_z\leq L_z$ on the event $\{ \D_{n-\rho} \subset 
	S(v_n) \}$. Therefore,
	\begin{multline}
		\label{eq:Nzbound}
		\Pr\biggl( \union_{z\in\L_n} \{ N_z>7\eta \} \biggr)
		= \Pr\biggl( \union_{z\in\L_n} \{ M_z-L_z>7\eta \} \biggr) \\
		\leq \sum_{z\in\L_n} \Pr( M_z-\tilde{L}_z > 7\eta ) + \Pr\left( 
		\D_{n-4\sqrt{n\log n}}\not\subset S(v_n) \right).
	\end{multline}

	To obtain a bound on $\Pr( M_z-\tilde{L}_z > 7\eta )$, note that
	\[
		\Ex M_z = 2n(n+1)\,\Pr_o( X(\tau_n)=z ) = \frac{n+1}{2}.
	\]
	Moreover, by Lemma~\ref{lem:uniformhitting}
	\[
		\sum_{y \in \L_k} \Pr(\tau_z^y=\tau_n^y)
		= 4k \, \Pr_k(\tau_z=\tau_n) = \frac{k}{n},
	\]
	hence
	\[
		\Ex\tilde{L}_z = \sum_{k=1}^{n-\rho} \frac{k}{n}
		= \frac{n+1}{2} - \rho + \frac{\rho(\rho-1)}{2n}.
	\]
	In particular, $\Ex M_z - \Ex\tilde{L}_z < \rho-1 \leq 4\eta$ for large 
	enough~$n$, so that
	\begin{equation}
	\label{eq:MzLz}
	\begin{split}
		\Pr( M_z-\tilde{L}_z > 7\eta )
		&\leq \Pr( M_z-\tilde{L}_z > \Ex M_z - \Ex\tilde{L}_z + 3\eta ) \\
		&\leq \Pr\bigl( M_z > \Ex M_z + \tfrac32\eta \quad\text{or}\quad 
		\tilde{L}_z < \Ex\tilde{L}_z - \tfrac32\eta \bigr) \\
		&\leq \Pr\bigl( M_z > \Ex M_z + \tfrac32\eta \bigr) + \Pr\bigl( 
		\tilde{L}_z < \Ex\tilde{L}_z - \tfrac32\eta \bigr).
	\end{split}
	\end{equation}
	By Lemma~\ref{lem:Chernoff},
	\[
		\Pr\bigl( \tilde{L}_z < \Ex\tilde{L}_z - \tfrac32\eta \bigr)
		\leq \exp\left( -\frac12 \frac{(3\eta/2)^2}{\Ex\tilde{L}_z} \right)
		< \exp\left( -\frac{9}{8}\frac{\eta^2}{n/2} \right)
		= n^{-9/4}.
	\]
	Likewise, for sufficiently large~$n$
	\[
	\begin{split}
		\Pr\bigl( M_z > \Ex M_z + \tfrac32\eta \bigr)
		&\leq \exp\left( -\frac12\frac{(3\eta/2)^2}{\Ex M_z+\eta/2} \right) \\
		&= \exp \left( -\frac94 \frac{n\log n}{n+1+\sqrt{n\log n}} \right) \\
		&< 2n^{-9/4}.
	\end{split}
	\]
	Combining \eqref{eq:Nzbound}, \eqref{eq:MzLz} and Lemma~\ref{lem:Sbound} 
	yields for sufficiently large~$n$
	\[
		\Pr\biggl( \union_{z\in\L_n} \{ N_z>7\eta \} \biggr)
		< 3n^{-9/4}\#\L_n + 6 n^{-2} < 13n^{-5/4}. \qedhere
	\]
\end{proof}

Given random sets $A,B \subset \Z^2$, we write $A \eqindist B$ to mean that 
$A$ and $B$ have the same distribution.

\begin{lemma}
	\label{lem:AinsideE}
	Let $m = \ceil{29n\sqrt{n\log n}}$. For all sufficiently large~$n$, there 
	exist random sets $A' \eqindist A(v_n)$ and $E' \eqindist E(m)$ such that
	\[
		\Pr\bigl( A' \not\subset E' \bigr) < 14n^{-5/4}.
	\]
\end{lemma}

\begin{proof}
	By the abelian property, Lemma~\ref{abelianproperty}, we can obtain 
	$A(v_n)$ from the stopped cluster~$S(v_n)$ by starting $N_z$ particles at 
	each $z \in \L_n$, and letting all but one of them walk until finding an 
	unoccupied site. More formally, let $x_1 = o$ and $x_{i+1} = 
	Y^i(\sigma_S^i)$ for $0<i<v_n$. Then
	\[
		\# \{i \leq v_n \,:\, x_i=z\} = \begin{cases}
			N_z, & z \in \L_n \\
			1, & z \in S(v_n)-\L_n \\
			0, & \mbox{else}
		\end{cases}
	\]
	and
	\[ A(v_n) \eqindist A(x_1,\ldots,x_{v_n}). \]

	To build up the extended cluster~$E(m)$ in a similar fashion, let $s = 
	v_n+m$, and let $y_1,\ldots,y_s \in \Z^2$ be such that	
	$\{y_1,\ldots,y_{v_n}\} = \D_n$, and
	\[ y_{v_n+i} = Y^{v_n+i-1}(\tau_n^{v_n+i-1}), \qquad i=1,2,\ldots,m. \]
	By Lemma~\ref{abelianproperty}, we have
	\[ E(m) \eqindist A(y_1,\ldots,y_s). \]

	For each $z \in \L_n$, let
	\[
		\tilde{N}_z = \sum\nolimits_{0\leq i<m} \I \left\{ \tau_z^{v_n+i} = 
		\tau_n^{v_n+i} \right\}
	\]
	be the number of extended particles that first hit layer~$\L_n$ at~$z$.  
	Then
	\[
		\#\{i \leq s \,:\, y_i = z\} = \begin{cases}
			\tilde{N}_z, & z \in \L_n \\
			1, & z \in \D_{n-1} \\
			0, & \mbox{else}.
		\end{cases}
	\]
	
	Now	let $A' = A(x_1,\ldots,x_{v_n})$ and consider the event
	\[
		\mathcal{E} = \inter_{z\in\L_n} \bigl\{ N_z \leq \tilde{N}_z \bigr\}.
	\]
	By Lemma~\ref{monotonicity}, on the event~$\mathcal{E}$ there exists a 
	random set $E' \eqindist A(y_1,\ldots,y_s)$ such that $A'\subset E'$. 
	Therefore, to finish the proof it suffices to show that $\Pr( 
	\mathcal{E}^c )< 14n^{-5/4}$.
%On $\mathcal{A}$, there exists a permutation $y'_1,\ldots,y'_s$ of 
%$y_1,\ldots,y_s$ such that 	\[ y'_i = x_i, \qquad i=1,\ldots,v_n. \]
%(On $\mathcal{A}^c$, set $y'_i=y_i$ for all $i=1,\ldots,s$.)  Let \[ A' = 
%A(x_1,\ldots,x_{v_n}) \eqindist A(v_n) \]
%and \[ E' = A(y'_1,\ldots,y'_s) \eqindist E(m). \]
%	\[ E' \eqindist A(y_1,\ldots,y_s) \eqindist E(m). \]
%Since the event $\{A' \subset E'\}$ contains $\mathcal{A}$, to finish the 
%proof it suffices to show that $\Pr(\mathcal{A}^c)<14n^{-5/4}$.
	Note that $\tilde{N}_z$ is a sum of independent indicators, and
	\[
		\Ex \tilde{N}_z = \frac{m}{4n} \geq \frac{29}{4} \eta
	\]
	where $\eta := \sqrt{n \log n}$. Setting $b=\eta/4$ in 	
	Lemma~\ref{lem:Chernoff} yields for sufficiently large~$n$
	\[
		\Pr\left(\tilde{N}_z \leq 7\eta \right)
		\leq \exp\left( -\frac12 \frac{b^2}{\Ex\tilde{N}_z} \right)
		= \exp\left( -\frac{1}{232} \sqrt{n\log n} \right) < \frac14 n^{-9/4},
	\]
	hence by Lemma~\ref{lem:Nzbound}
	\[
		\begin{split}
		\Pr(\mathcal{E}^c)
		&\leq \Pr\biggl( \bigcup_{z \in \L_n} \bigl\{ N_z > 7\eta 
		\quad\text{or}\quad \tilde{N}_z \leq 7\eta \bigr\} \biggr) \\
		& \leq \Pr\biggl( \union_{z\in\L_n} \bigl\{ N_z > 7\eta \bigr\} 
		\biggr) + \sum_{z\in\L_n} \Pr\bigl( \tilde{N}_z \leq 7\eta \bigr) \\
		&< 14n^{-5/4}. \qedhere
		\end{split}
	\]
\end{proof}

To finish the argument it remains to show that with high probability, the 
extended cluster $E(m)$ is contained in a slightly larger diamond. Here we 
follow the strategy used in the proof of the outer bound in~\cite{LBG92}.

\begin{lemma}
	\label{lem:Ebound}
	Let $m = \ceil{29n\sqrt{n \log n}}$. For all sufficiently large~$n$,
	\[
		\Pr\left( E(m) \not\subset \D_{n+20\sqrt{n\log n}} \right) < n^{-2}.
	\]
\end{lemma}

\begin{proof}
	For $j,k \geq 1$, let
	\[
		Z_k(j) = \#\bigl( E(j) \cap \L_{n+k} \bigr)
	\]
	and let $\mu_k(j) = \Ex Z_k(j)$. Then $\mu_k(j)$ is the expected number of 
	particles that have attached to the cluster in layer $\L_{n+k}$ after the 
	first $j$ extended particles have aggregated. Note that
	\[
		\mu_k(i+1)-\mu_k(i)
		= \Pr\bigl( Y^{v_n+i+1}(\sigma_E^{i+1}) \in \L_{n+k} \bigr).
	\]
	By property~(U1), in order for the $(i+1)^{\rm th}$ extended particle to 
	attach to the cluster in layer~$\L_{n+k}$, it must be inside the 
	cluster~$E(i)$ when it first reaches layer~$\L_{n+k-1}$. Therefore, by 
	Lemma~\ref{lem:uniformhitting},
	\[
		\begin{split}
		\mu_k(i+1)-\mu_k(i)
		&\leq \Pr\bigl( Y^{v_n+i+1}(\tau^{v_n+i+1}_{n+k-1}) \in E(i) \bigr) \\
		&= \sum_{y\in\L_{n+k-1}} \Pr_o\bigl( 
		Y^{v_n+i+1}(\tau_{n+k-1}^{v_n+i+1}\bigr) = y ) \cdot \Pr\bigl( y\in 
		E(i) \bigr) \\
%		&= \sum_{y\in\L_{n+k-1}} \Pr_o( X(\tau_{n+k-1}) = y ) \cdot \Pr\bigl( 
%		y\in E(i) \bigr) \\
		&= \frac{1}{4(n+k-1)} \cdot \mu_{k-1}(i)
		\leq \frac{\mu_{k-1}(i)}{4n}.
		\end{split}
	\]
	Since $\mu_k(0)=0$, summing over $i$ yields
	\[
		\mu_k(j) \leq \frac{1}{4n} \sum_{i=1}^{j-1} \mu_{k-1}(i).
	\]
	Since $\mu_1(j)\leq j$ and $\sum_{i=1}^{j-1} i^{k-1} \leq j^k/k$, we 
	obtain by induction on~$k$
	\[
		\mu_k(j)
		\leq 4n \left( \frac{j}{4n} \right)^k \frac{1}{k!}
		\leq 4n \left( \frac{je}{4nk} \right)^k,
	\]
	where in the last equality we have used the fact that $k!\geq k^ke^{-k}$.  
	Since $29e/80<1$, setting $j = m$ and $k = \floor{20\sqrt{n\log n}}$ we 
	obtain
	\[
		\mu_{k+1}(m) \leq 4n\left( \frac{\ceil{29 n\sqrt{n\log n}}e}{4n \cdot 
		20\sqrt{n\log n}} \right)^{k+1} < n^{-2}
	\]
	for sufficiently large~$n$. To complete the proof, note that
	\[
		\Pr( E(m) \not\subset \D_{n+k} )
		= \Pr( Z_{k+1}(m) \geq1 ) \leq \mu_{k+1}(m). \qedhere
	\]
\end{proof}

\begin{proof}[Proof of Theorem~\ref{thm:diamondshape}]
	Write $\eta = \sqrt{n \log n}$.	Since $S(v_n) \subset A(v_n)$, we have by 
	Lemma~\ref{lem:Sbound}
	\[
		\sum_{n\geq1} \Pr\bigl( \D_{n-4\eta} \not\subset A(v_n) \bigr)
		\leq \sum_{n\geq1} \Pr\bigl( \D_{n-4\eta} \not\subset S(v_n) \bigr)
		< \infty.
	\]
	Likewise, by Lemmas~\ref{lem:AinsideE} and~\ref{lem:Ebound}
	\[
		\begin{split}
		\sum_{n\geq1} \Pr\bigl( A(v_n) \not\subset \D_{n+20\eta} \bigr)
		&\leq \sum_{n \geq 1} \Pr\bigl( A(v_n) \not\subset E(m) \bigr)
		+ \sum_{n \geq 1} \Pr\bigl( E(m) \not\subset \D_{n+20\eta} \bigr) \\
		&< \infty.
		\end{split}
	\]
	By Borel-Cantelli we obtain Theorem~\ref{thm:diamondshape}.
\end{proof}

\section{The inward directed case}
\label{sec:inward}

\begin{proof}[Proof of Theorem~\ref{thm:inward}]
	Write $\ell = n-\ceil{6\log_r n}$, and denote by
	\[
		\mathcal{A}_n
		= \inter_{0<i<v_n} \inter_{z\in\D_\ell} \{ \tau^i_z < \tau^i_n \}
	\]
	the event that each of the first $v_n-1$ walks visits every site $z\in 
	\D_\ell$ before hitting layer~$\L_n$. Since $\#\D_{n-1} < v_n$, at least 
	one of the first $v_n-1$ particles must exit $\D_{n-1}$ before aggregating 
	to the cluster: $\sigma^i \geq \tau_n^i$ for some $i<v_n$. On the event 
	$\mathcal{A}_n$, this particle visits every site $z\in \D_\ell$ before 
	aggregating to the cluster, so $\D_\ell \subset A(i) \subset A(v_n)$.  
	Hence
	\[
		\Pr\bigl( \mathcal{D}_\ell\not\subset A(v_n) \bigr)
		\leq \Pr( \mathcal{A}_n^c ) \\
		\leq \sum_{0<i<v_n} \sum_{z\in\D_\ell} \Pr(\tau_z^i\geq\tau_n^i)
	\]
	By Lemma~\ref{lem:youcantavoidz},
	\[
		\begin{split}
		\Pr\bigl( \mathcal{D}_\ell\not\subset A(v_n) \bigr)
		&< 2n(n+1) \sum_{k=1}^{\ell} 4k (4k-1) r^{k-n} \\
		&\leq 32 n^3(n+1) \frac{r^{\ell+1}-r}{r^n(r-1)} \\
		&\leq \frac{32r}{r-1} n^3(n+1) \cdot n^{-6},
		\end{split}
	\]
	and by Borel-Cantelli we conclude that $\Pr( \D_\ell \subset A(v_n) \text{ 
	eventually} ) = 1$.
	
	Likewise, writing $m = n+\floor{6\log_r n}$, let
	\[
		\mathcal{B}_n
		= \inter_{0<i<v_n} \inter_{z\in\D_n} \{ \tau^i_z<\tau^i_m \}
	\]
	be the event that each of the first $v_n-1$ walks visits every site 
	$z\in\D_n$ before hitting layer~$\L_m$. Since the occupied cluster 
	$A(v_n-1)$ has cardinality $v_n-1 = \#\D_n-1$, there is at least one site 
	$z\in\D_n$ belonging to $A(v_n-1)^c$. On the event $\mathcal{B}_n$, each 
	of the first $v_n-1$ particles visits $z$ before hitting layer~$\L_m$, so
	\[
		\sigma^i \leq \tau_z^i < \tau_m^i, \qquad i=1,\dotsc,v_n-1.
	\]
	Therefore,
	\[
		\Pr\bigl( A(v_n)\not\subset \D_m \bigr)
		\leq \Pr( \mathcal{B}_n^c ) \\
		\leq \sum_{0<i<v_n} \sum_{z\in\D_n} \Pr(\tau_z^i\geq\tau_m^i).
	\]
	By Lemma~\ref{lem:youcantavoidz},
	\[
		\begin{split}
		\Pr\bigl( A(v_n)\not\subset \D_m \bigr)
		&< 2n(n+1) \sum_{k=1}^{n} 4k (4k-1) r^{k-m} \\
		&\leq 32 n^3(n+1) \frac{r^{n+1}-r}{r^m(r-1)} \\
		&\leq \frac{32r^2}{r-1} n^3(n+1) \cdot n^{-6},
		\end{split}
	\]
	and by Borel-Cantelli we conclude that $\Pr( A(v_n) \subset \D_m \text{ 
	eventually} ) = 1$.
\end{proof}
	
\section{The outward directed case}
\label{sec:outward}

To prove Theorem~\ref{thm:outward} we make use of a specific property of the 
uniformly layered walks for $p=0$. Recall that these walks have 
transition kernel~$\Qout$. By \eqref{eq:Qoutbegin}--\eqref{eq:Qoutend}, such a 
walk can only reach the site $(m,0)$ for $m\geq1$ by visiting the sites 
$(0,0),(1,0),\dotsc,(m,0)$ in turn. We can use this to find the exact growth 
rate of the clusters $A(i)$ along the $x$-axis.

Suppose that we count time according to the number of particles we have added 
to the growing cluster, and for $m\geq1$ set
\[
	T_m := \min\{ n\geq0: (m,0)\in A(n+1)\}.
\]
Then we can interpret $T_m$ as the time it takes before the site $(m,0)$ 
becomes occupied. The following lemma gives the exact order of the 
fluctuations in~$T_m$ as $m\to\infty$.

\begin{lemma}
	\label{lem:outwardaxis}
	For $p=0$ we have that
	\[
		\Pr\left( \limsup_{m\to\infty}
			\frac{T_m-2m(m+1)}{\sqrt{32(m^3 \log\log m)/3}} = 1
		\right) = 1
	\]
	and
	\[
		\Pr\left( \liminf_{m\to\infty}
			\frac{T_m-2m(m+1)}{\sqrt{32(m^3 \log\log m)/3}} = -1
		\right) = 1.
	\]
\end{lemma}

\begin{proof}
	Set $X_1=T_1$ and $X_m = T_m-T_{m-1}$ for $m>1$. Consider the aggregate at 
	time~$T_{m-1}$ when $(m-1,0)$ gets occupied. Since a walk must follow the 
	$x$-axis to reach the site $(m-1,0)$, we know that at time~$T_{m-1}$ all 
	sites $\{ (i,0) : i=0,1,\dotsc,m-1 \}$ are occupied and all sites $\{ 
	(i,0) : i\geq m\}$ are vacant. Now consider the additional time $X_m = 
	T_m-T_{m-1}$ taken before the site $(m,0)$ becomes occupied. Each walk 
	visits $(m,0)$ if and only if it passes through the sites $(1,0), (2,0), 
	\dotsc, (m,0)$ during the first~$m$ steps, which happens with probability 
	$1/4m$. Thus $X_m$ has the geometric distribution with parameter~$1/4m$. 
	Moreover, the $X_i$ are independent. Hence $T_m$ is a sum of independent 
	geometric random variables~$X_i$.

	Since $\Ex X_i = 4i$, $\Var X_i =16i^2-4i$ and $\Ex X_i^3 = 384i^3 - 96i^2 
	+ 4i$,
	\[
		B_m = \sum_{1\leq i\leq m} \Var X_i = \frac{16}{3} m^3 + O(m^2)
	\]
	and
	\[
		\sum_{1\leq i\leq m} \Ex\bigl( |X_i-\Ex X_i|^3 \bigr)
		\leq \sum_{1\leq i\leq m} \bigl(\Ex X_i^3 + (\Ex X_i)^3 \bigr)
		= O(m^4).
	\]
	By Lemma~\ref{lem:Esseen}, $\Delta_m = O(m^{-1/2})$, which shows that 
	Petrov's conditions of Lemma~\ref{lem:Petrov} are satisfied. Therefore,
	\[
		\Pr\left( \limsup_{m\to\infty}
			\frac{T_m-\Ex T_m}{\sqrt{2B_m\log\log B_m}} = 1
		\right) = 1.
	\]
	Since $\Ex T_m = \sum_{i=1}^m 4i = 2m(m+1)$ and $B_m = 16m^3/3+O(m^2)$, 
	this proves the first statement in Lemma~\ref{lem:outwardaxis}. The second 
	statement is obtained by applying Lemma~\ref{lem:Petrov} to $-T_m = 
	\sum_{i=1}^m (-X_i)$.
\end{proof}

\begin{proof}[Proof of Theorem~\ref{thm:outward}]
	Fix $\eps>0$, set $\eta := \sqrt{2(n\log\log n)/3}$ and let $\rho = 
	\ceil{(1-\eps)\eta}$. If we write $m = n-\rho$, then
	\begin{align*}
		2m(m+1) &= 2n(n+1) - 4n\rho + o(n), \\
		\sqrt{32(m^3 \log\log m)/3} &= 4n \eta + o(n^{5/4}\log\log n).
	\end{align*}
	Hence, setting $m = n-\rho$ in Lemma~\ref{lem:outwardaxis} gives
	\[
		\Pr\left( \limsup_{n\to\infty}
			\frac{T_{n-\rho}-2n(n+1)+4n\rho}{4n\eta} = 1
		\right) = 1.
	\]
	Since $\{(n-\rho,0) \not\in A(v_n)\} = \{T_{n-\rho} > v_n-1\}$ and $v_n - 
	1 = 2n(n+1)$, this implies
	\[
		\Pr\bigl( (n-\rho,0) \not\in A(v_n) \text{ i.o.} \bigr) = 1.
	\]
	Likewise, setting $m = n+\rho$ in Lemma~\ref{lem:outwardaxis} gives
	\[
		\Pr\left( \liminf_{n\to\infty}
			\frac{T_{n+\rho}-2n(n+1)-4n\rho}{4n\eta} = -1
		\right) = 1,
	\]
	hence
	\[
		\Pr\bigl( (n+\rho,0) \in A(v_n) \text{ i.o.} \bigr) = 1.
		\qedhere
	\]
\end{proof}

\section{Concluding Remarks}

\begin{figure}
\begin{center}
	\begin{tabular}{ccc}
	&&
	\includegraphics[height=.1\textheight]{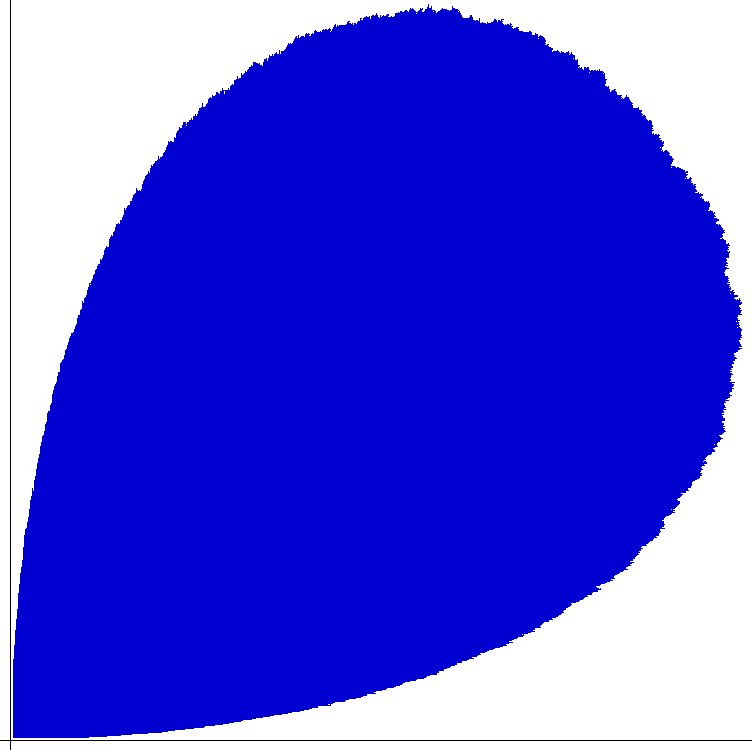} \\
	&& (3,3) \medskip \\
	&
	\includegraphics[height=.1\textheight]{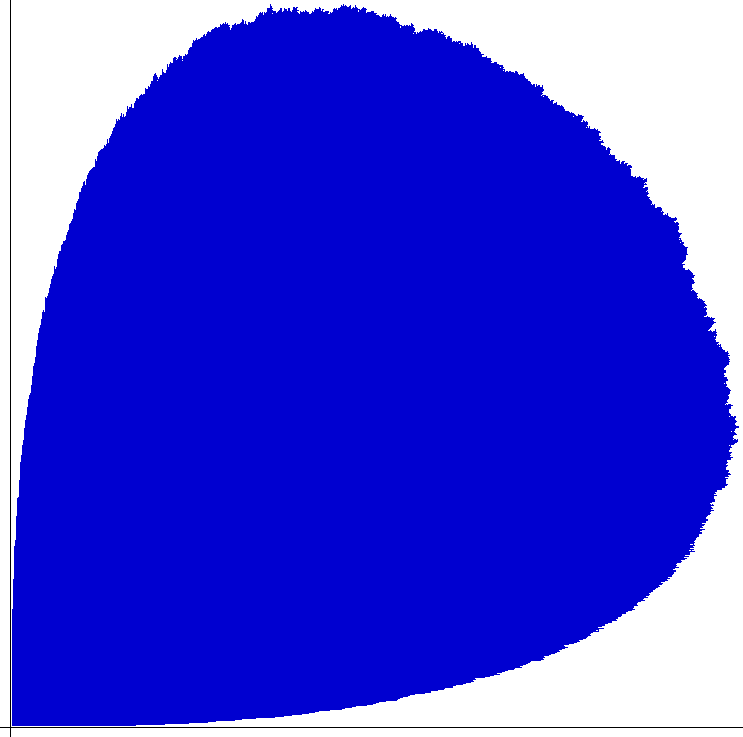} &
	\includegraphics[height=.1\textheight]{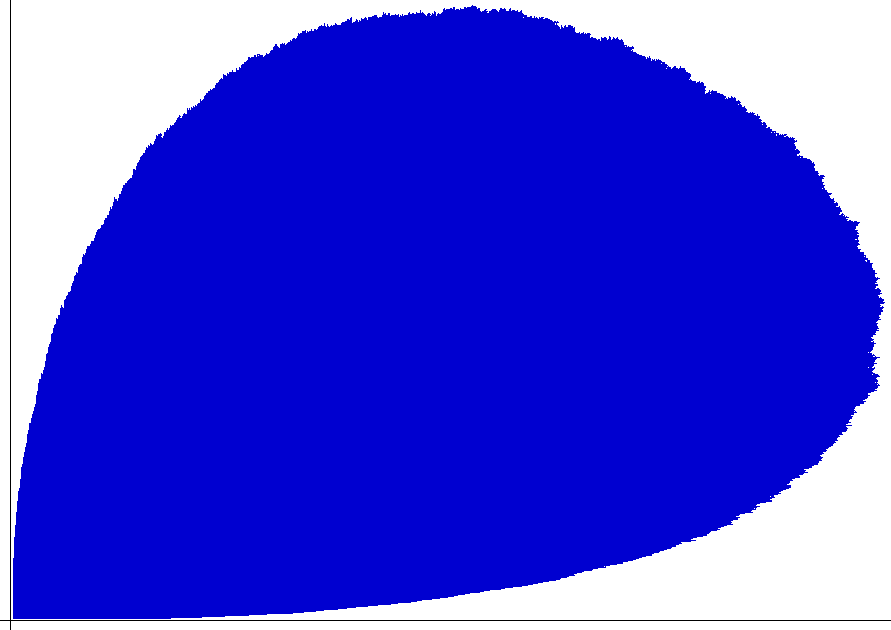} \\
	& (2,2) & (3,2) \medskip \\
	\includegraphics[height=.1\textheight]{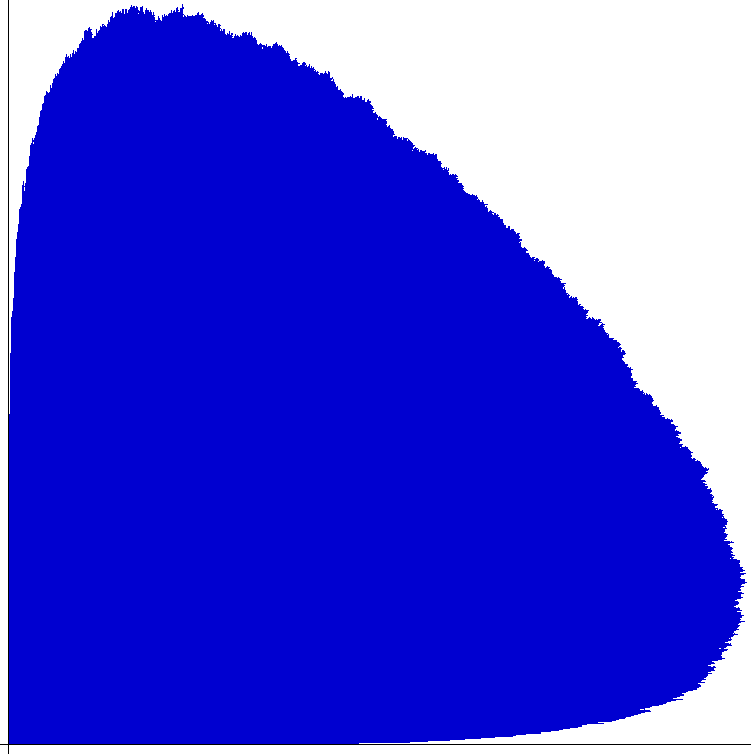} &
	\includegraphics[height=.1\textheight]{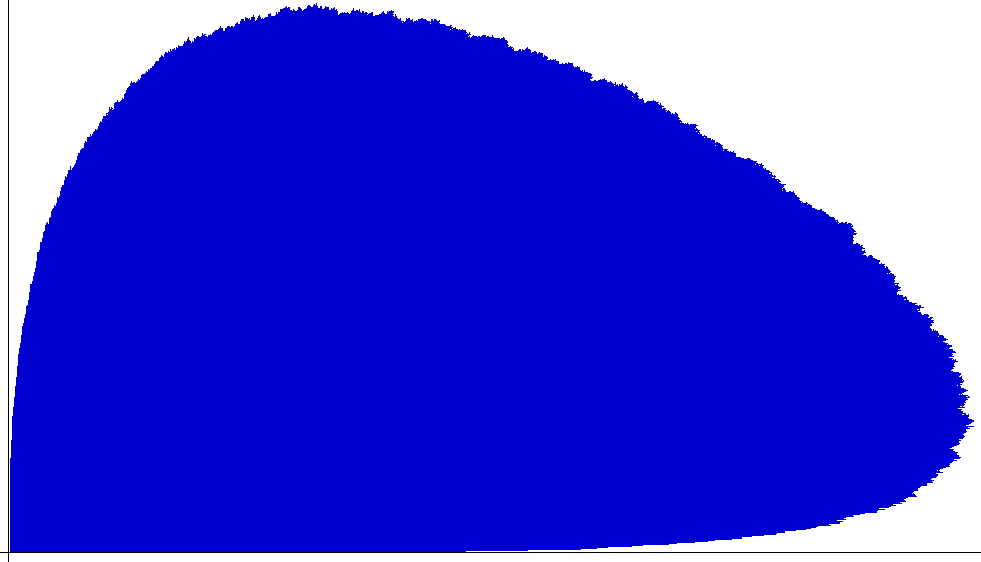} &
	\includegraphics[height=.1\textheight]{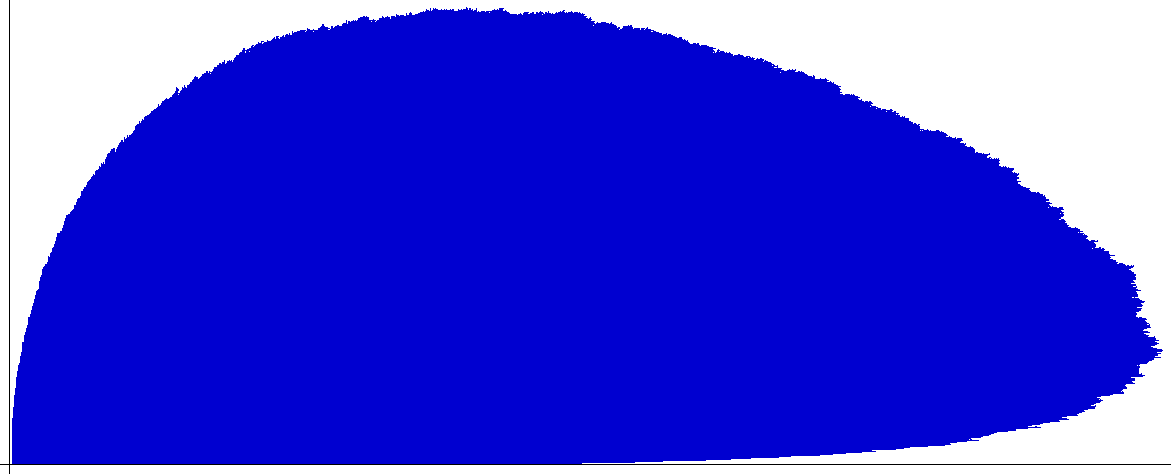} \\
	(1,1) & (2,1) & (3,1)
	\end{tabular}
\end{center}
\caption{Internal DLA clusters in the first quadrant of $\Z^2$ based on the 
outward-directed layered walk $Q_{out}$ started from a point other than the 
origin. For example, the cluster on the lower left is formed from $405\,900$ 
particles started at the point $(1,1)$.}
\label{Fig:off-center}
\end{figure}

Lawler, Bramson and Griffeath \cite{LBG92} discovered a key property of the 
Euclidean ball that characterizes it as the limiting shape of internal DLA 
clusters based on simple random walk in~$\Z^d$: for simple random walk killed 
on exiting the ball, any point~$z$ sufficiently far from the boundary of the 
ball is visited more often in expectation by a walk started at the origin than 
by a walk started at a uniform point in the ball. Uniformly layered walks have 
an analogous property with respect to the diamond: the Green's function 
$g(y,\cdot)$ for a walk started at~$y$ and killed on exiting $\D_n$ satisfies
\[
	g(o,z) \geq \frac{1}{\# \D_n} \sum_{y \in \D_n} g(y,z)
\]
for all $z\in \D_n$. Indeed, both the walk started at $o$ and the walk started 
at a uniform point in~$\D_n$ are uniformly distributed on layer 
$\L_{\norm{z}}$ at the time~$\tau_{\norm{z}}$ when they first hit this layer, 
so the expected number of visits to $z$ after time $\tau_{\norm{z}}$ is the 
same for both walks. The inequality comes from the fact that a walk started at 
the origin must hit layer $\L_{\norm{z}}$ before exiting~$\D_n$.

We conclude with two questions. The first concerns uniformly layered walks 
started from a point other than the origin. Figure~\ref{Fig:off-center} shows 
internal DLA clusters for six different starting points in the first quadrant 
of~$\Z^2$.  These clusters are all contained in the first quadrant. Our 
simulations indicate that a limiting shape exists for each starting point, and 
that no two starting points have the same limiting shape; but we do not know 
of any explicit characterization of the shapes arising in this way.

The second question is, do there exist walks with bounded increments having 
uniform harmonic measure on $L^1$ spheres in~$\Z^d$ for $d\geq 3$?

\section*{Acknowledgement}

We thank Ronald Meester for fruitful discussions.

\end{document}